\documentclass[reqno,11pt,oneside]{amsart}
\usepackage[utf8]{inputenc}
\usepackage{amsmath}
\usepackage{amsfonts}
\usepackage{amssymb}
\usepackage{graphicx}
\usepackage{verbatim}
\usepackage{mathrsfs}
\usepackage{upref,amsthm,amsxtra,exscale}
\usepackage{cite}
\usepackage{dsfont}
\usepackage[
  left=3.5cm,
  right=2.5cm,
  top=4cm,
  bottom=4cm
]{geometry}
\usepackage[colorlinks=true,urlcolor=blue,
citecolor=red,linkcolor=blue,linktocpage,pdfpagelabels,
bookmarksnumbered,bookmarksopen]{hyperref}

\usepackage{fullpage}

\newtheorem{theorem}{Theorem}[section]

\newtheorem{lemma}[theorem]{Lemma}

\numberwithin{equation}{section}

\def\bn{\mathbb{B}^n}
\def\D{\mathbb{B}^n_+}
\def\r{\mathbb{R}}
\def\rn{\mathbb{R}^n}
\def\hn{\mathbb{R}^n_+}
\def\z{\mathbb{Z}}
\def\z2{\mathbb{Z}_2}

\def\sn{\mathbb{S}^n}
\def\n{\mathbb{N}}
\def\cc{\mathbb{C}}
\def\eps{\varepsilon}
\def\rh{\rightharpoonup}

\def\irn{\int_{\rn}}
\def\ihn{\int_{\hn}}
\def\idhn{\int_{\partial\hn}}
\def\vp{\varphi}

\def\vr{\varrho}

\def\cC{\mathcal{C}}

\def\cN{\mathcal{N}}

\def\supp{\mathrm{supp}}
\def\what{\widehat}
\def\tilde{\widetilde}

\def\dist{\mathrm{dist}}

\begin{document}

\title[Sign-changing solutions to the Yamabe problem on a spherical cap]
{Sign-changing solutions to the  Yamabe problem on a spherical cap}

\author{Mónica Clapp}
\address[Mónica Clapp]{Instituto de Matemáticas,
Universidad Nacional Autónoma de México,
Campus Juriquilla,
76230 Querétaro, Qro., Mexico}
\email{monica.clapp@im.unam.mx}

\author{Benedetta Pellacci}
\address[Benedetta Pellacci]{Dipartimento di Matematica e Fisica,
Universit\`a della Campania ``Luigi Vanvitelli'',
Viale Lincoln 5,
81100 Caserta, Italy}
\email{benedetta.pellacci@unicampania.it}

\author{Angela Pistoia}
\address[Angela Pistoia]{Dipartimento SBAI, Sapienza Universit\`a di Roma,
via Antonio Scarpa 16, 00161 Roma, Italy }
\email{angela.pistoia@uniroma1.it}

\thanks{A. Pistoia  is partially supported by  the MUR-PRIN-20227HX33Z
  ``Pattern formation in nonlinear phenomena'' and  partially by
  INDAM-GNAMPA project ``Problemi di doppia curvatura su variet\`a a
  bordo e legami con le EDP di tipo ellittico''. 
 B. Pellacci is partially supported by  the MUR-PRIN-20227HX33Z
  ``Pattern formation in nonlinear phenomena'' and  partially by
  INDAM-GNAMPA project ``Problemi di ottimizzazione in PDEs da modelli biologici''.}


\begin{abstract}
Spherical caps play a crucial role in establishing a criterion for the existence of solutions to the Yamabe problem on a compact Riemannian manifold with boundary, similar to the role played by the standard sphere in the problem on a closed Riemannian manifold. This problem is expressed in terms of a nonlinear boundary-value problem, where both the nonlinearity and the boundary condition are critical in the Sobolev sense.

This work focuses on the existence of multiple solutions to the Yamabe problem on spherical caps. We show that if the spherical cap is contained in a hemisphere of the standard $n$-sphere and $n = 5$ or $n \geq 7$, the Yamabe problem has infinitely many sign-changing solutions.

Our approach takes advantage of symmetries and is based on a careful analysis of the loss of compactness of the variational problem.
\end{abstract}

\maketitle

\section{Introduction}

In this paper we study the critical problem
\begin{equation}\label{problem_cap}
\begin{cases}
 -\Delta u = a_n|u|^\frac{4}{n-2}u &\text{in \ } \hn,\\
 \frac{\partial u}{\partial x_n} = -b|u|^\frac{2}{n-2}u &\text{on \ }\partial\hn,
\end{cases}
\end{equation}
where $\hn:=\{(x_1,\ldots,x_n)\in\rn:x_n>0\}$, $a_n:=n(n-2)$, $n\geq 3$ and $b\in\r$.

This problem plays a crucial role in the question of whether a given compact Riemannian manifold $(M,g)$ with boundary has a metric conformally equivalent to $g$ that has constant scalar curvature in $M$ and for which the boundary has constant mean curvature. For a positive manifold $M$ this is equivalent to the existence of a positive solution to the Yamabe problem
\begin{equation}\label{yamabe problem}
\begin{cases}
 -\Delta_g u + c_nR_gu= a_n|u|^\frac{4}{n-2}u &\text{on \ } M,\\
 \frac{\partial u}{\partial\nu_g} + d_nh_gu= b|u|^\frac{2}{n-2}u &\text{on \ }\partial M,
\end{cases}
\end{equation}
where $\Delta_g$ is the Laplacian, $R_g$ is the scalar curvature of $M$, $h_g$ is the mean curvature of its boundary and $\nu_g$ is the outward unit normal on $\partial M$ with respect to the metric $g$, $c_n:=\frac{n-2}{4(n-1)}$ and  $d_n:=\frac{n-2}{2}$. A positive solution $u$ to \eqref{yamabe problem} gives rise to a conformal metric $\tilde{g}:=u^\frac{4}{n-2}g$ with the desired properties. In \cite{e1,e4} Escobar established the existence of a positive solution to \eqref{yamabe problem} for a large class of manifolds if $|b|$ is sufficiently small. Subsequently, Han and Li showed in \cite{hl0,hl} that, in many cases a positive solution exists for every $b\in\r$.

In \cite{hl} Han and Li gave a condition for the existence of a positive solution to \eqref{yamabe problem}, similar to that established by Aubin in \cite{aubin} for the Yamabe problem on a closed Riemannian manifold. The role of the standard sphere in Aubin's condition is played by the the spherical cap
$$B_b:=\pi_b^{-1}(\hn),$$
given by the stereographic projection $\pi_b$ from the north pole $N_b:=(0,\ldots,0,-\frac{b}{n-2},1)$ of the unit sphere in $\r^{n+1}$ centered at $\xi_b:=(0,\ldots,0,-\frac{b}{n-2},0)$; see Figure \ref{fig}.  
\begin{figure}[htbp]
        \centering
        \includegraphics[width=0.70\textwidth]{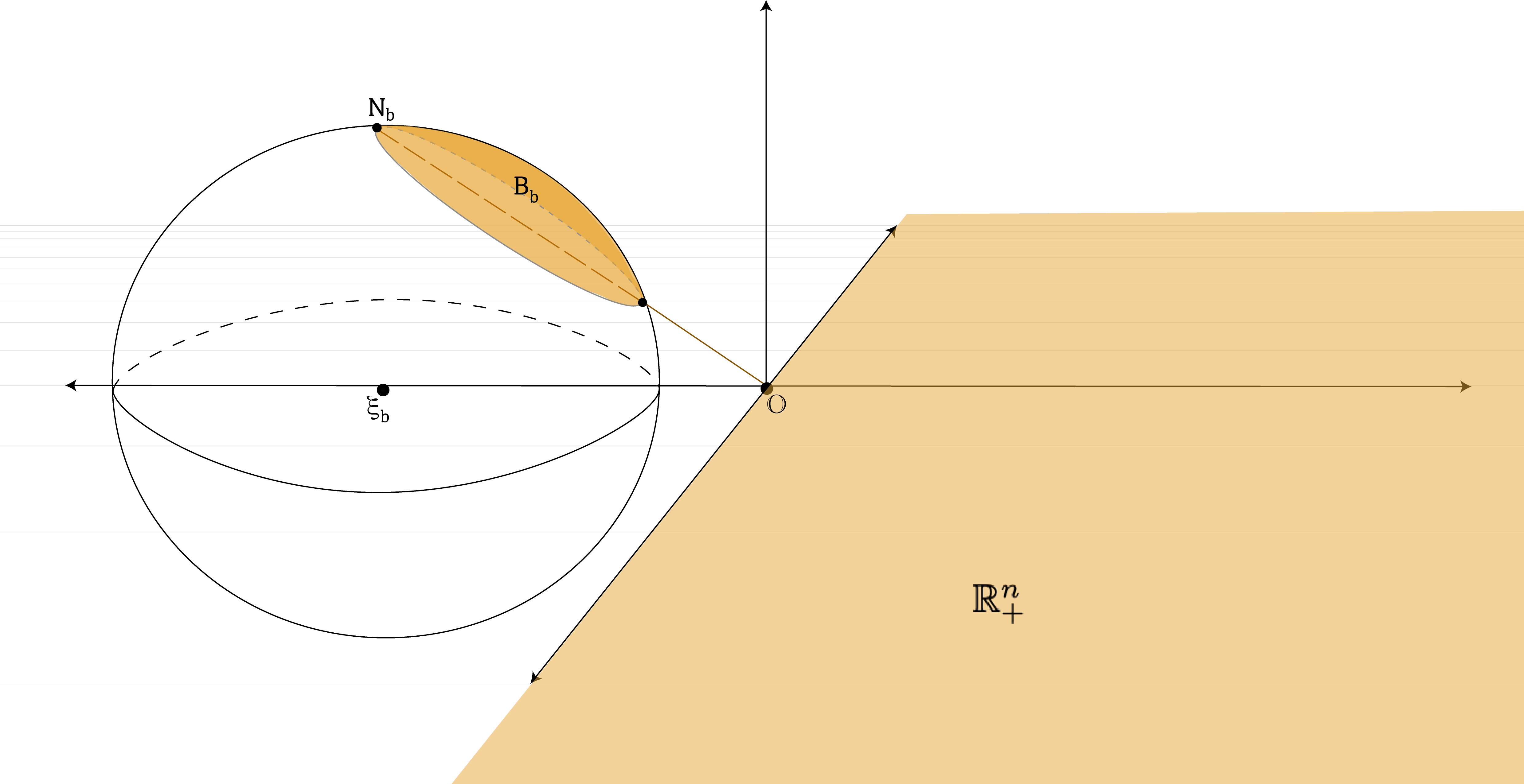}
        \caption{Spherical cap}
        \label{fig}
  \end{figure}
The spherical cap plays also a crucial role on the existence of a least energy sign-changing solution to \eqref{yamabe problem}, as was recently shown in \cite{cpp}. 

The stereographic projection $\pi_b$ establishes a one-to-one correspondence between the solutions of \eqref{problem_cap} and those of problem \eqref{yamabe problem} on $B_b$. The positive solutions to \eqref{problem_cap} are the \emph{bubble} 
$$U(x):=\left(\frac{1}{1+x_1^2+\cdots+x_{n-1}^2+(x_n+\frac{b}{n-2})^2}\right)^{(n-2)/2},\qquad x=(x_1,\ldots,x_n)\in\overline{\hn},$$
and its dilations $U_\eps(x):=\eps^\frac{2-n}{2}U\big(\frac{x}{\eps}\big)$, $\eps>0$. It is easy to see that \eqref{problem_cap} does not have a least energy sign-changing solution. Therefore, it is natural to ask whether it has higher energy sign-changing solutions. As far as we know, this question has not yet been answered.

Many results concerning the existence of nodal solutions to the Yamabe equation
\begin{equation}\label{yamabe rn}
-\Delta u = a_n|u|^\frac{4}{n-2}u \qquad\text{in \ } \rn
\end{equation}
in the whole Euclidean space are already well known. The first one was obtained by W.Y. Ding in \cite{ding}, where he established the existence of infinitely many sign-changing solutions taking advantage of the fact that there are groups of isometries on the sphere $\sn$ whose orbits are all infinite. These isometries translate into conformal diffeomorphisms of $\rn$ through the stereographic projection. Other types of nodal solutions to \eqref{yamabe rn} were subsequently exhibited in \cite{dmpp1,dmpp2,c,cfs,fp,mm}.

Ding's method cannot be used for our problem because the spherical caps are not invariant under the action of the groups that he considered. In \cite{dmpp1,dmpp2} del Pino, Musso, Pacard and Pistoia obtained another type of solutions using the Lyapunov-Schmidt reduction method. Their approach takes advantage of the invariance of \eqref{yamabe rn} under the Kelvin transform, which is not available for problem \eqref{problem_cap}.

However, as we will explain below, it is possible to adapt some ideas used in \cite{c} to prove the following result.

\begin{theorem} \label{thm:main_cap}
If $n=5$ or $n\geq 7$, then there exists a decreasing sequence $(\beta_k)$ of positive numbers such that, if $b\geq -\beta_k$, the problem \eqref{problem_cap} has at least $k$ different sign-changing solutions.

In particular, \eqref{problem_cap} has infinitely many sign-changing solutions for every $b\geq 0$.
\end{theorem}

Our strategy is as follows. The main difficulty in addressing problem \eqref{problem_cap} using variational methods lies in the lack of compactness produced by its invariance under dilations. To maintain control over the blow-up points we use groups of linear isometries whose fixed-point space is the $x_n$-axis and whose remaining orbits have infinite cardinality. Furthermore, we require that the group admits a surjective homomorphism $\phi$ onto the group $\{1,-1\}$. This allows us to generate functions that are sign-changing by construction, which are called $\phi$-equivariant. Such groups exist only in dimensions $n=5$ and $n\geq 7$.

Next, we introduce an auxiliary mixed boundary-value problem \eqref{problem_d} on the half-ball, which does not have a least-energy $\phi$-equivariant solution, and we analyze the behavior of minimizing sequences for it. We observe two different phenomena. Concentration can occur either on the boundary, in which case the blow-up profile is a least-energy $\phi$-equivariant solution to \eqref{problem_cap}, or it occurs in the interior of the half-ball, in which case the blow-up profile is a least-energy $\phi$-equivariant solution of the Yamabe problem \eqref{yamabe rn} in the whole space. A detailed description is given in Theorem \ref{thm:minimizing}. Furthermore, we see that for $b\geq 0$ and for $b$ negative sufficiently close to $0$ only the first behavior is possible and therefore a $\phi$-equivariant solution for problem \eqref{problem_cap} must exist; see Theorem \ref{thm:existence}. Considering different symmetries we obtain the multiplicity result.

We note that the existence of a least-energy $\phi$-equivariant solution to the Yamabe problem \eqref{yamabe rn} in $\rn$ was already proven in \cite{c}. Using a result from that paper we obtain the following improvement of Theorem \ref{thm:main_cap} for $b=0$.

\begin{theorem} \label{thm:main_n=6}
If $n\geq 5$ and $b=0$, the problem \eqref{problem_cap} has infinitely many sign-changing solutions.
\end{theorem}

The existence of a sign-changing solution to the problem \eqref{problem_cap} when $b < 0$ remains an interesting and challenging open problem.

To complete the picture, it is worth mentioning that, for the scalar flat problem obtained by replacing $a_n$ with $0$ and taking $b=1$ in the equations \eqref{problem_cap}, Almaraz and Wang \cite{aw} proved the existence of infinitely many sign-changing solutions for $n\geq 4$.

The paper is organized as follows. In Section \ref{sec:setting} we present the symmetric variational setting and compare the least energy level of the problem \eqref{problem_cap} with that of \eqref{yamabe rn} in that setting. In Section \ref{sec:D} we introduce the mixed boundary-value auxiliary problem, we describe the behavior of its $\phi$-equivariant minimizing sequences, and derive an abstract existence result for problem \eqref{problem_cap}. Finally, in Section \ref{sec:proof} we prove Theorems \ref{thm:main_cap} and \ref{thm:main_n=6}.

\section{The symmetric variational setting}
\label{sec:setting}

Set $p:=\frac{2n}{n-2}$ and $q:=\frac{2(n-1)}{n-2}$. Let $D^{1,2}(\hn):=\{u\in L^p(\hn):\nabla u\in L^2(\hn,\rn)\}$ with its usual norm
\begin{equation}\label{eq:norm}
\|u\|:=\Big(\ihn|\nabla u|^2\Big)^\frac{1}{2}.
\end{equation}
The solutions to the problem \eqref{problem_cap} are the critical points of the $\cC^2$-functional $J_b:D^{1,2}(\hn)\to\r$ given by
\begin{equation}\label{eq:J}
J_b(u):=\frac{1}{2}\ihn|\nabla u|^2 - \frac{a_n}{p}\ihn|u|^p - \frac{b}{q}\idhn |u|^q,
\end{equation}
whose derivative at $u$ is
\begin{equation}\label{eq:derivative}
J'_b(u)v=\ihn\nabla u\cdot\nabla v - a_n\ihn |u|^{p-2}uv - b\idhn |u|^{q-2}uv,\qquad v\in D^{1,2}(\hn).
\end{equation}
We are interested in critical points that possess a particular type of symmetries, which we describe below.

Let $G$ be a closed subgroup of the group $O(n-1)$ of linear isometries of $\r^{n-1}$ and let $\phi:G\to\mathbb{Z}_2=\{1,-1\}$ be a continuous homomorphism of groups. We denote the $G$-orbit of the point $y\in\r^{n-1}$ by
$$Gy:=\{gy:g\in G\}$$
and its cardinality by $\# Gy$. We assume that $G$ and $\phi$ have the following properties:
\begin{itemize}
\item[$(G_1)$] $\# Gy=\infty$ for every $y\in\r^{n-1}\smallsetminus\{0\}$.
\item[$(G_2)$] There exists $y_0\in\r^{n-1}$ such that $Gy_0\neq Ky_0$, where $K:=\ker\phi$.
\end{itemize}
We extend the action of $G$ to $\rn$ by setting
$$gx:=(gx',t)\qquad\text{where \ }x=(x',t)\in\r^{n-1}\times\r\equiv\rn.$$
The fixed-point space of this action is the set
$$\{x\in\rn:gx=x\text{ \ for all \ }g\in G\}=\{(0,t):t\in\r\},$$
and the $G$-orbit of every point $x\in\rn$ that does not belong to this set has infinite cardinality. A subset $X$ of $\rn$ is said to be $G$-invariant if $Gx\subset X$ for every $x\in X$, and a function $u:X\to\r$ will be called \emph{$\phi$-equivariant} if 
\begin{equation}
u(gx)=\phi(g)u(x)\quad\text{for every \ }g\in G, \ x\in X.
\end{equation}
Property $(G_2)$ implies, in particular, that the homomorphism $\phi$ is surjective. Therefore, every nontrivial $\phi$-equivariant function changes sign. 

We are interested in finding $\phi$-equivariant solutions to the problem \eqref{problem_cap}. Note that the half-space $\hn$ is $G$-invariant. Property $(G_2)$ guarantees that the space
$$D^{1,2}(\hn)^\phi:=\{u\in D^{1,2}(\hn):u\text{ \ is \ }\phi\text{-equivariant}\}$$
has infinite dimension. It is the fixed-point space of $D^{1,2}(\hn)$ under the action of $G$ given by $(gu)(x):=\phi(g)u(g^{-1}x)$. Hence, by the principle of symmetric criticality \cite[Theorem 1.28]{w}, the $\phi$-equivariant solutions to problem \eqref{problem_cap} are the critical points of the restriction of the functional $J_b$ defined above, to the space $D^{1,2}(\hn)^\phi$. Abusing notation, we denote this restriction by $J_b$.

The nontrivial critical points of $J_b:D^{1,2}(\hn)^\phi\to\r$ belong to the set
\begin{align*}
\cN_b^\phi(\hn)&:=\{u\in D^{1,2}(\hn)^\phi: u\neq 0 \text{ and }J'(u)u=0\} \\
&\,=\Big\{u\in D^{1,2}(\hn)^\phi: u\neq 0 \text{ and }\|u\|^2=F_b(u)\Big\},
\end{align*}
where 
\begin{equation}\label{eq:F}
F_b(u):=a_n\ihn |u|^p + b\idhn|u|^q.
\end{equation}

\begin{lemma}\label{lem:nehari}
\begin{itemize}
\item[$(a)$]There exists $c_0>0$ such that $\|u\|\geq c_0$ for all $u\in\cN^\phi_b(\hn)$.
\item[$(b)$]$\cN^\phi_b(\hn)$ is a Hilbert submanifold of class $\cC^2$ of $D^{1,2}(\hn)^\phi$ and a natural constraint for $J_b$.
\item[$(c)$]If $u\in D^{1,2}(\hn)\smallsetminus\{0\}$, then there exists a unique $t_u\in(0,\infty)$ such that $t_uu\in\cN^\phi_b(\hn)$. The function $t\mapsto J(tu)$ is strictly increasing in $[0,t_u]$ and strictly decreasing in $[t_u,\infty)$.
\end{itemize}
\end{lemma}

\begin{proof}
This is proved with the same argument used in \cite[Lemma 2.1]{cpp}.
\end{proof}

Set
$$\mu^\phi_b(\hn):=\inf_{u\in\cN^\phi_b(\hn)}J_b(u).$$
Next, we compare this value with the corresponding one for the Yamabe problem in the whole space
\begin{equation}\label{problem_rn}
\begin{cases}
 -\Delta u = a_n|u|^\frac{4}{n-2}u,\\
 u\in D^{1,2}(\rn)^\phi,
\end{cases}
\end{equation}
where $D^{1,2}(\rn)^\phi:=\{u\in D^{1,2}(\rn):u\text{ \ is \ }\phi\text{-equivariant}\}$. We write $J_\infty:D^{1,2}(\rn)^\phi\to\r$ for the functional
$$J_\infty(u):=\frac{1}{2}\irn|\nabla u|^2 - \frac{a_n}{p}\irn|u|^p,$$
associated to this problem, and set 
$$\cN_\infty^\phi:=\{u\in D^{1,2}(\rn)^\phi: u\neq 0 \text{ and }J'_\infty(u)u=0\} $$ 
and
$$\mu^\phi_\infty:=\inf_{u\in\cN^\phi_\infty}J_\infty(u).$$

\begin{lemma} \label{lem:comparison}
Let $G$ and $\phi$ satisfy $(G_1)$ and $(G_2)$. Then, the following hold.
\begin{itemize}
\item[$(i)$] There exists $\bar{\omega}\in\cN^\phi_0(\hn)$ such that $\bar{\omega}$ solves \eqref{problem_cap} for $b=0$ and $2J_0(\bar{\omega})=\mu^\phi_\infty$.
\item[$(ii)$] If
\begin{equation}\label{eq:b}
b\geq -\frac{2(n-1)}{n(2n-3)}\frac{\ihn|\nabla \bar{\omega}|^2}{\idhn |\bar{\omega}|^q},
\end{equation}
then $\mu^\phi_\infty>\mu^\phi_b(\hn)$. 
\end{itemize}
\end{lemma}

\begin{proof}
$(i):$ \ It is shown in \cite[Corollary 3.4]{c} that $\mu^\phi_\infty$ is attained at a function $\omega\in\cN^\phi_\infty$ that satisfies $\omega(x',t)=\omega(x',-t)$ for every $(x',t)\in\r^{n-1}\times\r$. Hence, $\frac{\partial\omega}{\partial t}(x',0)=0$. Therefore, its restriction $\bar{\omega}:=\omega|_{\hn}$ to $\hn$ solves \eqref{problem_cap} for $b=0$. So it belongs to $\cN_0^\phi(\hn)$ and 
$$\mu^\phi_\infty=J_\infty(\omega)=2J_0(\bar{\omega}).$$
$(ii):$ \ Set
$$\bar{c}:=\ihn|\nabla \bar{\omega}|^2 = {a_n}\ihn|\bar{\omega}|^p\qquad\text{and}\qquad \bar{d}:=\idhn |\bar{\omega}|^q$$
and let $t\in(0,\infty)$ be such that $t\bar{\omega}\in\cN^\phi_b(\hn)$. Then,
\begin{equation}\label{eq:t}
\bar{c}t^2=\bar{c}t^p+b\bar{d}t^q.
\end{equation}
It follows that
\begin{align*}
2J_0(t\bar{\omega})-J_b(t\bar{\omega})&=\frac{1}{2}\bar{c}t^2-\frac{1}{p}\bar{c}t^p+b\frac{1}{q}\bar{d}t^q=\Big(\frac{1}{2}-\frac{1}{p}\Big)\bar{c}t^p+b\Big(\frac{1}{2}+\frac{1}{q}\Big)\bar{d}t^q \\
&=t^q\Big(\frac{p-2}{2p}\bar{c}t^{p-q} + b\Big(\frac{q+2}{2q}\Big)\bar{d}\Big).
\end{align*}
If $b\geq 0$, this equality yields $2J_0(t\bar{\omega})>J_b(t\bar{\omega})$. On the other hand, if $b<0$, then \eqref{eq:t} implies that $t>1$ and using assumption \eqref{eq:b} we obtain
$$\frac{p-2}{2p}\bar{c}t^{p-q} + b\Big(\frac{q+2}{2q}\Big)\bar{d}>\frac{1}{n}\bar{c} + b\Big(\frac{2n-3}{2n-2}\Big)\bar{d}\geq0.$$
Therefore, $2J_0(t\bar{\omega})>J_b(t\bar{\omega})$ if $b$ satisfies \eqref{eq:b}. Statement $(i)$ and Lemma \ref{lem:nehari}$(c)$ yield
$$\mu_\infty^\phi=2J_0(\bar{\omega}) \geq 2J_0(t\bar{\omega}) >J_b(t\bar{\omega})\geq \mu^\phi_b(\hn),$$
as claimed.
\end{proof}

\section{An auxiliary problem}
\label{sec:D}

Throughout this section we assume that $G$ and $\phi$ are as described in Section \ref{sec:setting} and satisfy $(G_1)$ and $(G_2)$.

Let $\bn$ be the open unit ball in $\rn$ and
$$\D:=\{x\in\hn:|x|<1\}=\bn\cap\hn$$
be the open half-ball in $\hn$. This set is $G$-invariant. We denote by
$$\Gamma_0:=\{x\in \partial\D:|x|=1\}\qquad\text{and}\qquad \Gamma_1:=\partial\D\cap\partial\hn$$
the inner and outer parts of its boundary, and consider the mixed boundary-value problem

\begin{equation}\label{problem_d}
\begin{cases}
 -\Delta u = a_n|u|^\frac{4}{n-2}u &\text{in \ } \D,\\
 u=0 &\text{on \ }\Gamma_0,\\
\frac{\partial u}{\partial x_n} = -b|u|^\frac{2}{n-2}u &\text{on \ }\Gamma_1.
\end{cases}
\end{equation}
Let $V(\D)$ be the space of functions in $D^{1,2}(\D)$ whose trace vanishes on $\Gamma_0$. Note that $V(\D)\subset D^{1,2}(\hn)$ via trivial extension. Abusing notation, we denote by $J_b$ the restriction to $V(\D)$ of the functional defined in \eqref{eq:J}. Set
$$\cN_b^\phi(\D):=\cN_b^\phi(\hn)\cap V(\D)\qquad\text{and}\qquad \mu^\phi_b(\D):=\inf_{u\in\cN^\phi_b(\D)}J_b(u).$$

\begin{lemma} \label{lem:no_minimum}
$\mu^\phi_b(\D)=\mu^\phi_b(\hn)$ and $\mu^\phi_b(\D)$ is not attained by $J_b$ on $\cN_b^\phi(\D)$.
\end{lemma}

\begin{proof}
As $\cN_b^\phi(\D)\subset\cN_b^\phi(\hn)$, we have that $\mu^\phi_b(\hn)\leq\mu^\phi_b(\D)$. To prove the opposite inequality, let $\vp_k\in\cN_b^\phi(\hn)\cap\cC^\infty(\overline{\hn})$ be such that $\vp_k$ has compact support and $J_b(\vp_k)\to\mu^\phi_b(\hn)$. Choose $\eps_k>0$ such that the support of $\bar{\vp}_k(x):=\eps_k^{(2-n)/2}\vp_k(\eps_k^{-1}x)$ is contained in $\overline{\D}\smallsetminus\overline{\Gamma}_0$. Then $\bar{\vp}_k\in\cN_b^\phi(\D)$ and $\mu^\phi_b(\D)\leq J_b(\bar{\vp}_k)=J_b(\vp_k)$ for all $k$. Therefore, $\mu^\phi_b(\D)\leq\mu^\phi_b(\hn)$.

If $\mu^\phi_b(\D)$ were attained by $J_b$ at some $u\in\cN_b^\phi(\D)$, the trivial extension of $u$ to $\hn$ would be a solution to \eqref{problem_cap}, contradicting the unique continuation principle.
\end{proof}

Our next goal is to describe the behavior of minimizing sequences for $J_b$ on $\cN_b^\phi(\D)$. The following lemma highlights the role of assumption $(G_1)$.

\begin{lemma}\label{lem:G}
If $G$ satisfies $(G_1)$ then, any given sequences $(\eps_k)$ in $(0,\infty)$ and $(y_k)$ in $\r^{n-1}$ contain subsequences such that one of the following statements holds:
\begin{enumerate}
\item[$(i)$] either there exists $C_0>0$ such that $\eps_{k}^{-1}|y_k|<C_0$ for all $k\in\n$,
\item[$(ii)$] or, for each $m\in\n$, there exist $g_1,\ldots,g_m\in G$ such that $\eps_{k}^{-1}|g_iy_k-g_jy_k|\to\infty$ as $k\to\infty$ for any $i\neq j$, $1\leq i,j\leq m$.
\end{enumerate}
\end{lemma}

\begin{proof}
After passing to a subsequence, we have that $\eps_k^{-1}|y_k|\to a\in[0,\infty]$. 

If $a<\infty$, then $(i)$ holds true. 

If $a=\infty$, passing to a subsequence we have that $y_k\neq 0$ for all $k\in\n$ and
$$\frac{y_k}{|y_k|}\to y\quad\text{in \ }\r^{n-1}.$$
By $(G_1)$, for any given $m\in\n$ there exist $g_1,\ldots,g_m\in G$ such that $g_iy\neq g_jy$ if $i\neq j$. Let $3d:=\min_{i\neq j}|g_iy-g_jy|$, and let $k_0\in\n$ be such that
$\big|\frac{y_k}{|y_k|}-y\big|<d$ if $k\geq k_0$. Then,
\begin{align*}
d\leq \Big|\frac{g_iy_k}{|y_k|}-\frac{g_jy_k}{|y_k|}\Big|\qquad\text{if \ }i\neq j\text{ \ and \ }k\geq k_0.
\end{align*}
Therefore, $\eps_k^{-1}|g_iy_k-g_jy_k|\geq\eps_k^{-1}|y_k|d\to\infty$ if $i\neq j$. This shows that, if $a=\infty$, then $(ii)$ holds true.
\end{proof}

\begin{theorem} \label{thm:minimizing}
Let $(u_k)$ be a sequence in $\cN_b^\phi(\D)$ such that $J_b(u_k)\to \mu^\phi_b(\D)$. Then, after passing to a subsequence, one of the following statements holds true:
\begin{itemize}
\item[$(I)$] There exists a sequence of positive numbers $(\eps_k)$ and a nontrivial solution $w\in D^{1,2}(\hn)^\phi$ to the problem \eqref{problem_cap} such that
$$\lim_{k\to\infty}\left\|u_k-\eps_k^\frac{2-n}{2}w\left(\frac{\cdot}{\eps_k}\right)\right\| = 0$$
and $J_b(w)=\mu^\phi_b(\hn)=\mu^\phi_b(\D)$.

\item[$(II)$] There is a sequence of positive numbers $(\eps_k)$, a sequence of points $\xi_k=(0,t_k)$ in $\D$ and a nontrivial solution $w\in D^{1,2}(\rn)^\phi$ to the problem \eqref{problem_rn} such that
$$\eps_k^{-1}\dist(\xi_k,\partial\D)\to\infty,\qquad\lim_{k\to\infty}\left\|u_k-\eps_k^\frac{2-n}{2}w\left(\frac{\cdot - \xi_k}{\eps_k}\right)\right\| = 0$$
and $\mu^\phi_\infty\leq J_\infty(w)=\mu^\phi_b(\D)$.
\end{itemize}
\end{theorem}

\begin{proof}
In what follows, we consider $V(\D)\subset D^{1,2}(\hn)$ by means of trivial extension. 

Since $u_k\in\cN_b^\phi(\D)$ we have that
$$J_b(u_k)=\frac{1}{2}\ihn|\nabla u_k|^2 - \frac{a_n}{p}\ihn|u_k|^p - \frac{b}{q}\idhn |u_k|^q\geq\frac{1}{2}\|u_k\|^2 - \frac{1}{q}F_b(u_k)=\frac{q-2}{2q}\|u_k\|^2,$$
with $\|\cdot\|$ and $F_b$ as defined in \eqref{eq:norm} and \eqref{eq:F}. Since $J_b(u_k)\to \mu^\phi_b(\D)$, the sequence $(u_k)$ is bounded and, passing to a subsequence, $u_k\rh u$ weakly in $V(\D)^\phi$. Using Ekeland's variational principle, we may assume that $J'_b(u_k)\to 0$ in $V(\D)'$. Then, for every $\vp\in \cC^\infty_c(\bn)$,
$$o(1)=J'_b(u_k)\vp=\int_{\D}\nabla u_k\cdot\nabla \vp - a_n\int_{\D} |u_k|^{p-2}u_k\vp - b\int_{\partial\D} |u_k|^{q-2}u_k\vp,$$
and, passing to the limit, yields $J_b'(u)=0$ in $V(\D)'$. We claim that $u=0$. Indeed, using the Brezis-Lieb lemma we see that
\begin{align}
J_b(u_k)&=J_b(u_k-u)+J_b(u)+o(1), \label{eq:D1}\\
J'_b(u_k)u_k&=J'_b(u_k-u)[u_k-u]+J'_b(u)u+o(1)=J'_b(u_k-u)[u_k-u]+o(1),\nonumber
\end{align}
and, as a consequence,
\begin{equation}\label{eq:D2}
J_b(u_k-u)\geq\frac{1}{2}\|u_k-u\|^2-\frac{1}{q}F_b(u_k-u)=\frac{q-2}{2q}\|u_k-u\|^2+o(1).
\end{equation}
Equations \eqref{eq:D1} and \eqref{eq:D2} imply that 
\begin{equation}\label{eq:D3}
J_b(u_k)\geq J_b(u)+o(1).
\end{equation}
So, if $u\neq 0$, we would have that $u\in\cN_b^\phi(\D)$ and $J_b(u)=\mu_1(\D)$, contradicting Lemma \ref{lem:no_minimum}. This shows that $u_k\rh 0$ weakly in $V(\D)$. 

Consider the concentration function
$$Q_k(r):=\sup _{x\in\rn}\Big\{\int_{B_r(x)\cap\D}|u_k|^p+\int_{B_r(x)\cap\Gamma_1}|u_k|^q\Big\},$$
where $B_r(x)$ is the open ball of radius $r$ centered at $x$ in $\rn$. By Lemma \ref{lem:nehari}$(a)$, 
$$0<c_0\leq\|u_k\|^2=F_b(u_k)+o(1)\leq a_n\int_{\D}|u_k|^p+\max\{b,0\}\int_{\Gamma_1}|u_k|^q+o(1).$$ 
Hence, there exists $\kappa_0>0$ such that, after passing to a subsequence, 
\begin{equation}\label{eq:kappa0}
0<\kappa_0<\int_{\D}|u_k|^p+\int_{\Gamma_1}|u_k|^q\qquad\text{for all \ }k\in\n.
\end{equation}
Let $\kappa\in (0,\min\{1,\kappa_0\})$, to be fixed later. Then, there exist $\eps_k>0$ and $x_k\in\rn$ such that
\begin{equation} \label{eq:kappa}
\kappa=Q_k(\eps_k)=\int_{B_{\eps_k}(x_k)\cap\D}|u_k|^p+\int_{B_{\eps_k}(x_k)\cap\Gamma_1}|u_k|^q.
\end{equation}
We claim that $\eps_k\to 0$. 

To prove this claim, we argue by contradiction. Assume there is $0<\what{\eps}\leq\eps_k$ for all $k$ and let $\vp\in\cC^\infty_c(B_{\what{\eps}}(z))$ for some $z\in\rn$. Since $u_k\to0$ strongly in $L^2(\D)$, using \eqref{eq:kappa} we obtain
\begin{align} \label{eq:eps_to_0}
&\int_{\D}|\nabla(\vp u_k)|^2=\int_{\D}\nabla u_k\cdot\nabla(\vp^2 u_k) + o(1) \nonumber \\
&=a_n\int_{\D}\vp^2|u_k|^p+b\int_{\Gamma_1}\vp^2|u_k|^q+o(1)  \nonumber \\
&\leq C_{n,b}\Big(\int_{\D}|u_k|^{p-2}|\vp u_k|^2+\int_{\Gamma_1}|u_k|^{q-2}|\vp u_k|^2\Big)+o(1) \nonumber  \\
&\leq C_{n,b}\left[\Big(\int_{\D\cap B_{\what{\eps}}(z)}|u_k|^p\Big)^\frac{p-2}{p}\Big(\int_{\D}|\vp u_k|^p\Big)^\frac{2}{p} + \Big(\int_{\Gamma_1\cap B_{\what{\eps}}(z)}|u_k|^q\Big)^\frac{q-2}{q}\Big(\int_{\Gamma_1}|\vp u_k|^q\Big)^\frac{2}{q}\right] + o(1) \nonumber  \\
&\leq \overline{C}_{n,b} \Big[\kappa^\frac{2}{n} + \kappa^\frac{1}{n-1}\Big]\int_{\D}|\nabla(\vp u_k)|^2 + o(1).
\end{align}
So choosing $\kappa$ sufficiently small we get that $\int_{\D}|\nabla(\vp u_k)|^2 = o(1)$ and, hence, that
$$\int_{\D}|\vp u_k|^p+\int_{\Gamma_1}|\vp u_k|^q=o(1)\qquad\text{for all \ }\vp\in\cC^\infty_c(B_{\what{\eps}}(z)), \ z\in\rn.$$
Taking $z_1,\ldots,z_\ell\in\rn$ such that $\D\subset B_\frac{\what\eps}{2}(z_1)\cup\cdots\cup B_\frac{\what\eps}{2}(z_\ell)$ and $\vp_i\in\cC^\infty_c(B_{\what{\eps}}(z_i))$ such that $\vp_i(z)=1$ if $|z-z_i|\leq\frac{\what\eps}{2}$ yields
$$\int_{\D}|u_k|^p+\int_{\Gamma_1}|u_k|^q\leq\sum_{i=1}^\ell\Big(\int_{\D}|\vp_iu_k|^p+\int_{\Gamma_1}|\vp_iu_k|^q\Big)=o(1),$$
which contradicts \eqref{eq:kappa0}. This shows that $\eps_k\to 0$.

Our next task is to replace the points $x_k$ with more convenient concentration points $\xi_k$. To this end, first, we 
write $x_k=(y_k,t_k)\in\r^{n-1}\times\r$ and recall that Lemma \ref{lem:G} gives two alternatives. We claim that the alternative $(ii)$ is impossible. Indeed, arguing by contradiction, assume that, for each $m\in\n$, there exist $g_1,\ldots,g_m\in G$ such that $\eps_{k}^{-1}|g_iy_k-g_jy_k|\to\infty$ as $k\to\infty$ for any $i\neq j$. Then, for $k$ large enough, $B_{\eps_k}(g_ix_k)\cap B_{\eps_k}(g_jx_k)=\emptyset$ and, as a consequence,
\begin{equation*}
m\kappa=\sum_{i=1}^m\Big(\int_{B_{\eps_k}(g_ix_k)\cap\D}|u_k|^p+\int_{B_{\eps_k}(g_ix_k)\cap\Gamma_1}|u_k|^q\Big)\leq\int_{\D}|u_k|^p+\int_{\Gamma_1}|u_k|^q
\end{equation*}
which is impossible because $(u_k)$ is bounded in $V(\D)$. Hence, the alternative $(i)$ must hold, that is, after passing to a subsequence, there exists $C_0>0$ such that $\eps_k^{-1}|y_k|<C_0$ for all $k\in\n$. Set $\zeta_k=(0,t_k)\in\r^{n-1}\times\r$. Then, \eqref{eq:kappa} yields
\begin{equation} \label{eq:kappa1}
\kappa\leq \int_{B_{(C_0+1)\eps_k}(\zeta_k)\cap\D}|u_k|^p+\int_{B_{(C_0+1)\eps_k}(\zeta_k)\cap\Gamma_1}|u_k|^q.
\end{equation}
Now we claim that, after passing to a subsequence, there exist points $\xi_k$ and $C_1>1$ such that
\begin{equation} \label{eq:kappa2}
\kappa\leq \int_{B_{C_1\eps_k}(\xi_k)\cap\D}|u_k|^p+\int_{B_{C_1\eps_k}(\xi_k)\cap\Gamma_1}|u_k|^q
\end{equation}
and one of the following statements holds true:
\begin{enumerate}
\item $\xi_k=(0,0)$ for all $k\in\n$.
\item $\xi_k=(0,t_k)\in\D$ for all $k\in\n$, and $\eps_k^{-1}\dist(\xi_k,\partial\D)\to\infty$.
\item $\xi_k=(0,1)$ for all $k\in\n$.
\end{enumerate}
Indeed, after passing to a subsequence there are three possibilities: If $(\eps_k^{-1}|\zeta_k|)$ is bounded, we set $\xi_k:=(0,0)$. Then, \eqref{eq:kappa2} follows from \eqref{eq:kappa1}. If $\eps_k^{-1}\dist(\zeta_k,\partial\D)\to\infty$ we set $\xi_k:=\zeta_k$. Since $\dist(\zeta_k,\D)\leq(C_0+1)\eps_k$ we have that $\zeta_k\in\D$ and \eqref{eq:kappa2} is simply \eqref{eq:kappa1}. Finally, if $(\eps_k^{-1}|\zeta_k|)$ is unbounded and $(\eps_k^{-1}|\zeta_k-(0,1)|)$ is bounded we set $\xi_k:=(0,1)$ and \eqref{eq:kappa2} follows from \eqref{eq:kappa1}.

Set
$$D_k:=\eps_k^{-1}(\D-\xi_k)\qquad\text{and}\qquad\what\Gamma_k:=\eps_k^{-1}(\Gamma_1-\xi_k),$$
and define
$$w_k(z)=\eps_k^\frac{n-2}{2}u_k(\eps_kz+\xi_k)\qquad\text{if \ }z\in D_k.$$
Since $\xi_k$ is a fixed-point of $G$, the function $w_k$ is $\phi$-equivariant, and it follows from \eqref{eq:kappa} and \eqref{eq:kappa2} that
\begin{align} \label{eq:kappa3}
\kappa =\sup_{x\in\rn}\Big\{\int_{B_1(z)\cap D_k}|w_k|^p+\int_{B_1(z)\cap\what\Gamma_k}|w_k|^q\Big\}\leq \int_{B_{C_1}(0)\cap D_k}|w_k|^p+\int_{B_{C_1}(0)\cap\what\Gamma_k}|w_k|^q.
\end{align}
Furthermore, for any $\vp\in\cC^\infty_c(\rn)$, setting $\vp_k(z):=\vp\big(\frac{x-\xi_k}{\eps_k}\big)$ we see that
\begin{equation}\label{eq:derivative1}
\int_{D_k}\nabla w_k\cdot\nabla(\vp^2w_k)-a_n\int_{D_k}\vp^2|w_k|^p-b\int_{\what\Gamma_k}\vp^2|w_k|^q = J_b'(u_k)[\vp_k^2u_k]=o(1).
\end{equation}
Next, we analize the behavior of $(w_k)$ in each of the three cases stated above.

\underline{Case $(1)$}: $\xi_k=(0,0)$ for all $k\in\n$.

In this case, $\bigcup_{k\geq 1}D_k=\hn$ and $w_k\in D^{1,2}(\hn)^\phi$. Since the sequence $(w_k)$ is bounded in $D^{1,2}(\hn)$, a subsequence satisfies $w_k\rh w$ weakly in $D^{1,2}(\hn)$, $w_k\to w$ a.e. in $\hn$, $w_k\to w$ in $L^2_{loc}(\hn)$, and $w$ is $\phi$-equivariant. We claim that $w\neq 0$. Otherwise, we get a contradiction as follows: Let $\vp\in\cC^\infty_c(B_1(z))$ for some $z\in\rn$. Note that $\supp(\vp)\cap\hn\subset D_k$ for every large enough $k$. So following the argument in \eqref{eq:eps_to_0}, this time using \eqref{eq:kappa3} and \eqref{eq:derivative1}, we conclude that 
$$\ihn|\vp w_k|^p+\idhn|\vp w_k|^q = o(1)\qquad\text{for every \ }\vp\in\cC^\infty_c(B_1(z)), \ z\in\rn.$$
Taking $z_1,\ldots,z_\ell\in\rn$ such that $B_{C_1}(0)\subset B_\frac{1}{2}(z_1)\cup\cdots\cup B_\frac{1}{2}(z_\ell)$ and $\vp_i\in\cC^\infty_c(B_1(z_i))$ such that $\vp_i(z)=1$ if $|z-z_i|\leq\frac{1}{2}$ we obtain
$$\int_{B_{C_1}(0)\cap \hn}|w_k|^p+\int_{B_{C_1}(0)\cap\partial\hn}|w_k|^q\leq\sum_{i=1}^\ell\Big(\ihn|\vp_iw_k|^p+\idhn|\vp_iw_k|^q\Big)=o(1),$$
which contradicts \eqref{eq:kappa3}. This shows that $w\neq 0$.

Furthermore, if $\vp\in\cC^\infty_c(\rn)$, then $\supp(\vp)\cap\hn\subset D_k$ for all large enough $k$, and setting $\vp_k(z):=\vp\big(\frac{x-\xi_k}{\eps_k}\big)$ we see that
\begin{equation*}
\ihn\nabla w_k\cdot\nabla\vp-a_n\ihn|w_k|^{p-2}w_k\vp-b\idhn|w_k|^{q-2}w_k\vp = J_b'(u_k)[\vp_k|_{\D}]=o(1).
\end{equation*}
Passing to the limit yields $J_b'(w)=0$ in $D^{1,2}(\hn)'$, that is, $w$ is a nontrivial $\phi$-equivariant solution to \eqref{problem_cap}. Now, using the Brezis-Lieb lemma and arguing  as we did at the beginning of the proof, we see that
\begin{align}\label{eq:H1}
\mu^\phi_b(\D)+o(1)&=J_b(w_k)=J_b(w_k-w)+J_b(w)+o(1)\geq J_b(w_k-w)+\mu^\phi_b(\hn)+o(1), \\
J_b(w_k-w)&\geq\frac{1}{2}\|w_k-w\|^2-\frac{1}{q}F_b(w_k-w)=\frac{q-2}{2q}\|w_k-w\|^2+o(1),
\end{align}
and, using Lemma \ref{lem:no_minimum}, we obtain  
$$\lim_{k\to\infty}\left\|u_k-\eps_k^\frac{2-n}{2}w\left(\frac{\cdot}{\eps_k}\right)\right\|=\lim_{k\to\infty}\|w_k-w\| = 0$$
and $J_b(w)=\mu^\phi_b(\D)=\mu^\phi_b(\hn)$. Hence, in this case, statement $(I)$ holds true.

\underline{Case $(2)$}: $\xi_k=(0,t_k)\in\D$ for all $k\in\n$, and $\eps_k^{-1}\dist(\xi_k,\partial\D)\to\infty$.

In this case, $\bigcup_{k\geq 1}D_k=\rn$. Consider the extension operator $D^{1,2}(\hn)\to D^{1,2}(\rn)$,  $v\mapsto v^*$, where
\begin{equation*}
v^*(x',t):=
\begin{cases}
v(x',t) &\text{if \ }t>0,\\
v(x',-t) &\text{if \ }t<0,
\end{cases}
\end{equation*}
and define $\what w_k\in D^{1,2}(\rn)$ by
$$\what w_k(z):=\eps_k^\frac{n-2}{2}u_k^*(\eps_kz+\xi_k).$$
Note that $\what w_k$ is $\phi$-equivariant and $\what w_k(z)=w_k(z)$ if $z\in D_k$. Since $(\what w_k)$ is bounded in $D^{1,2}(\rn)$, a subsequence satisfies $\what w_k\rh w$ weakly in $D^{1,2}(\rn)$, $\what w_k\to w$ a.e. in $\rn$, $\what w_k\to w$ in $L^2_{loc}(\rn)$, and $w$ is $\phi$-equivariant. If $\vp\in\cC^\infty_c(B_1(z))$ for some $z\in\rn$, then $\supp(\vp)\subset D_k$ for every large enough $k$ and following the argument in \eqref{eq:eps_to_0}, using \eqref{eq:kappa3} and \eqref{eq:derivative1}, we see that, if $w=0$, then
$$\irn|\vp w_k|^p = o(1)\qquad\text{for every \ }\vp\in\cC^\infty_c(B_1(z)), \ z\in\rn.$$
Since $B_{C_1}(0)\subset D_k$ for sufficiently large $k$, arguing as in Case (1) we  obtain
$$\int_{B_{C_1}(0)}|w_k|^p=o(1).$$
This contradicts \eqref{eq:kappa3} and shows that $w\neq 0$.

If $\vp\in\cC^\infty_c(\rn)$, then $\supp(\vp)\subset D_k$ for all large enough $k$, and setting $\vp_k(z):=\vp\big(\frac{x-\xi_k}{\eps_k}\big)$ we get that
\begin{align*}
\irn\nabla\what w_k\cdot\nabla\vp-a_n\irn|\what w_k|^{p-2}w_k\vp &=\irn\nabla w_k\cdot\nabla\vp-a_n\irn|w_k|^{p-2}w_k\vp \\
& = J_b'(u_k)[\vp_k]=o(1).
\end{align*}
Passing to the limit yields $J_\infty'(w)=0$ in $D^{1,2}(\rn)'$. Hence, $w$ is a nontrivial $\phi$-equivariant solution to \eqref{problem_rn}. 

Fix a radial cut-off function $\chi\in\cC_c^\infty(\rn)$ such that $\chi(z)=1$ if $|z|\leq 1$ and $\chi(z)=0$ if $|z|\geq 2$ and set
$$\chi_k(z):=\chi\Big(\frac{\eps_k z}{r_k}\Big)\qquad\text{where \ }r_k:=\frac{1}{2}\dist(\xi_k,\partial\D).$$
Note that $w\chi_k\in D^{1,2}_0(D_k)$ and
\begin{align*}
\irn|\nabla(w(\chi_k-1)|^2 &\leq C\left(\int_{|z|\geq\frac{r_k}{\eps_k}}|\nabla w|^2 + \Big(\frac{\eps_k}{r_k}\Big)^2\int_{\frac{r_k}{\eps_k}\leq|z|\leq\frac{2r_k}{\eps_k}}w^2\right) \\
&\leq C\left(\int_{|z|\geq\frac{r_k}{\eps_k}}|\nabla w|^2 + \Big(\int_{|z|\geq\frac{r_k}{\eps_k}}w^p\Big)^\frac{2}{p}\right) = o(1).
\end{align*}
Therefore,
\begin{align*}
\int_{D_k}|\nabla(w_k-w\chi_k)|^2 &=\int_{D_k}|\nabla w_k|^2 - 2\int_{D_k}\nabla w_k\cdot\nabla(w\chi_k) + \int_{D_k}|\nabla(w\chi_k)|^2  \\
&=\int_{D_k}|\nabla w_k|^2 - 2\irn\nabla \what w_k\cdot\nabla(w\chi_k) + \int_{D_k}|\nabla(w\chi_k)|^2  \\
&=\int_{D_k}|\nabla w_k|^2 - \irn|\nabla w|^2 + o(1).
\end{align*}
It is also easy to see that
$$\int_{D_k}|w_k-w\chi_k|^p = \int_{D_k}|w_k|^p - \irn|w|^p + o(1).$$
As a consequence,
\begin{align}\label{eq:J(v)}
J_b(u_k)&=\frac{1}{2}\int_{D_k}|\nabla w_k|^2-\frac{a_n}{p}\int_{D_k}|w_k|^p - \frac{b}{q}\int_{\partial\D}|u_k|^q \nonumber \\
&=\frac{1}{2}\int_{D_k}|\nabla(w_k-w\chi_k)|^2-\frac{a_n}{p}\int_{D_k}|\nabla(w_k-w\chi_k)|^p - \frac{b}{q}\int_{\partial\D}|v_k|^q \nonumber\\
&\qquad+ \frac{1}{2}\irn|\nabla w|^2-\frac{a_n}{p}\irn|w|^p + o(1) \nonumber\\
&=J_b(v_k) + J_\infty(w) + o(1),
\end{align}
where
$$v_k(x):=u_k(x)-\eps_k^\frac{2-n}{2}w\Big(\frac{x-\xi_k}{\eps_k}\Big)\chi\Big(\frac{x-\xi_k}{r_k}\Big).$$
Similarly,
\begin{equation}\label{eq:J'(v)}
o(1)=J_b'(u_k)u_k=J_b'(v_k)v_k + J_\infty'(w)w + o(1)=J_b'(v_k)v_k + o(1).
\end{equation}
Note that $v_k$ is $\phi$-equivariant. It follows from \eqref{eq:J'(v)} that $v_k\to 0$ strongly in $D^{1,2}(\hn)$, as, otherwise, $J_b(v_k)+o(1)\geq\mu_b^\phi(\D)$, contradicting \eqref{eq:J(v)}. Therefore, 
$$\lim_{k\to\infty}\left\|u_k-\eps_k^\frac{2-n}{2}w\Big(\frac{\cdot - \xi_k}{\eps_k}\Big)\right\| = \lim_{k\to\infty}\|v_k\| = 0$$
and $\mu^\phi_b(\D)=J_\infty(w)\geq \mu^\phi_\infty$. Hence, in this case, we obtain statement $(II)$.

\underline{Case $(3)$}: $\xi_k=(0,1)$ for all $k\in\n$.

In this case, $\bigcup_{k\geq 1}D_k=\mathbb{H}:=\{(x',t)\in\r^{n-1}\times\r:t<1\}$ and the functions 
$$\what w_k(z):=\eps_k^\frac{n-2}{2}u_k^*(\eps_kz+\xi_k)$$
belong to $D^{1,2}_0(\mathbb{H})$. A subsequence satisfies $\what w_k\rh w$ weakly in $D^{1,2}(\rn)$, $\what w_k\to w$ a.e. in $\rn$, $\what w_k\to w$ in $L^2_{loc}(\rn)$ and $w\in D^{1,2}_0(\mathbb{H})$. If $\vp\in\cC^\infty_c(B_1(z))$ for some $z\in\rn$, then $\supp(\vp)\subset \eps_k^{-1}(\hn-\xi_k)=\{(x',t)\in\r^{n-1}\times\r:t>-\frac{1}{\eps_k}\} $ for large enough $k$ and, arguing as in Case (2), we show that $w\neq 0$.

On the other hand, if $\vp\in\cC_c^\infty(\mathbb{H})$, then $\supp(\vp)\subset D_k$ for large enough $k$, and setting $\vp_k(z):=\vp\big(\frac{x-\xi_k}{\eps_k}\big)$ we obtain
\begin{equation*}
\int_\mathbb{H}\nabla w_k\cdot\nabla\vp-a_n\int_\mathbb{H}|w_k|^{p-2}w_k\vp = J_b'(u_k)[\vp_k]=o(1).
\end{equation*}
Passing to the limit shows that $w$ is a nontrivial solution to the problem 
\begin{equation*}
\begin{cases}
 -\Delta u = a_n|u|^\frac{4}{n-2}u,\\
 u\in D^{1,2}_0(\mathbb{H}),
\end{cases}
\end{equation*}
in the half-space $\mathbb{H}$. It is well known that this problem does not have a nontrivial solution. Therefore, Case (3) cannot occur.

This completes the proof.
\end{proof}

As a consequence we obtain the following result.

\begin{theorem} \label{thm:existence}
Let $G$ and $\phi$ satisfy $(G_1)$ and $(G_2)$. If $b$ satisfies \eqref{eq:b} then the problem \eqref{problem_cap} has a least energy $\phi$-equivariant solution.
\end{theorem}

\begin{proof}
Let $(u_k)$ be a sequence in $\cN_b^\phi(\D)$ such that $J_b(u_k)\to \mu^\phi_b(\D)$. By Lemmas \ref{lem:no_minimum} and \ref{lem:comparison}, $\mu^\phi_b(\D)=\mu^\phi_b(\hn)<\mu^\phi_\infty$. Hence statement $(I)$ in Theorem \ref{thm:minimizing} must be true. In particular, there exists a solution $w\in D^{1,2}(\hn)^\phi$ to the problem \eqref{problem_cap} such that
$J_b(w)=\mu^\phi_b(\hn)$.
\end{proof}

\section{Multiple nodal solutions}
\label{sec:proof}

We begin by giving some examples of $G$ and $\phi$ that satisfy $(G_1)$ and $(G_2)$.

For $n\geq 5$ we write $\r^{n-1}\equiv\cc^2\times\r^{n-5}$ and a point in $\r^{n-1}$ as $(z,y)$ with $z=(z_1,z_2)\in\cc\times\cc$ and $y\in\r^{n-5}$. 

For each $m\in\n$ we consider the group $\what G_m$ of linear isometries of $\cc^2$ generated by $\mathbb{S}^1\cup\{\vr_m\}$, where $\mathbb{S}^1$ is the group of unit complex numbers acting by 
\begin{equation*}
\zeta z:=(\zeta z_1,\overline{\zeta}z_2),\quad\text{for \ }\zeta\in \mathbb{S}^1, \ \ z=(z_1,z_2)\in\cc\times\cc,
\end{equation*}
and $\vr_m$ is given by 
\begin{equation*}\label{eq:rho}
\vr_mz:=\left(\Big(\cos\frac{\pi}{m}\Big)z+\Big(\sin\frac{\pi}{m}\Big)\tau z\right)\qquad\text{for \ }z\in\cc^2,
\end{equation*}
where $\tau(z_1,z_2):=(-\overline{z}_2,\overline{z}_1)$ and $\overline{\zeta}$ is the complex conjugate of $\zeta$. The $\what G_m$-orbit of any point $z\in\cc^2\smallsetminus\{0\}$ is a finite union of circles, so it has infinite cardinality.

Let $O(n-5)$ be the group of linear isometries of $\r^{n-5}$ and let $G_m:=\what G_m\times O(n-5)$ act on $\r^{n-1}$ by
\begin{equation} \label{eq:action}
gx':=(\gamma z,\theta y),
\end{equation}
where $g=(\gamma,\theta)\in \Gamma_m\times O(n-5)$ and $x'=(z,y)\in\cc^2\times\r^{n-5}$. Then, the $G_m$-orbit of $x'$ is 
$$G_mx'=\what G_mz\times S^{n-6}_{|y|},$$
where $S^{k-1}_r$ is the sphere of radius $r$ in $\r^k$, centered at the origin. Note that $G_m$ satisfies $(G_1)$ if and only if $n\neq 6$.

Now we define $\phi_m:G_m\to\mathbb{Z}_2=\{1,-1\}$ to be the homomorphism of groups given by
$$\phi_m(\gamma)=1 \text{ \ if \ }\gamma\in\mathbb{S}^1,\qquad \phi_m(\vr_m)=-1,\qquad\phi_m(\theta)=1\text{ \ if \ }\theta\in O(n-5).$$
Since $\vr_m$ has order $2m$ this homomorphism is well defined and it clearly satisfies $(G_2)$.

We note the following.

\begin{lemma}\label{lem:different}
If $\ell=jm$ with $j$ even and $u,v:\hn\to\r$ are nontrivial functions that satisfy
$$u(\vr_\ell x)=-u(x)\quad\text{and}\quad v(\vr_m x)=-v(x)\quad\text{for every \ }x\in\hn,$$
then $u\neq v$.
\end{lemma}

\begin{proof}
Arguing by contradiction, assume that $u=v$. Since $\vr_\ell^j=\vr_m$ and $j$ is even, we have that 
$$u(\vr_mx)=u(\vr_\ell^jx)=u(x)=v(x)=-v(\vr_mx)=-u(\vr_mx)$$
for every $x\in\hn$. This implies that $u$ and $v$ are trivial, contradicting our assumption.
\end{proof}
\medskip

\begin{proof}[Proof of Theorem \ref{thm:main_cap}]
For each $m$, we apply Lemma \ref{lem:comparison} to $G_m$ and $\phi_m$, as defined above, to obtain a 
function $\bar{\omega}_m\in\cN^{\phi_m}_0(\hn)$ such that $2J_0(\bar{\omega}_m)=\mu^{\phi_m}_\infty$, and we set
$$\beta(G_m,\phi_m):=\frac{2(n-1)}{n(2n-3)}\frac{\ihn|\nabla \bar{\omega}_m|^2}{\idhn |\bar{\omega}_m|^q}.$$
For each $k\in\n$ we define
$$\beta_k:=\min_{i=0,\ldots,k-1}\beta(G_{2^i},\phi_{2^i}).$$

If $b\geq -\beta_k$, then $b$ satisfies \eqref{eq:b} for each $G_{2^i}$, $\phi_{2^i}$ with $i=0,\ldots,k-1$, and Theorem \ref{thm:existence} yields a least energy $\phi_{2^i}$-equivariant solution $w_{2^i}$ to the problem \eqref{problem_cap}. This solution changes sign by construction. 
By Lemma \ref{lem:different}, all functions $w_1,\,w_2,\ldots,w_{2^{k-1}}$ are different.
\end{proof}
\medskip

\begin{proof}[Proof of Theorem \ref{thm:main_n=6}]
For each $m\geq 5$, we consider the action of the group $\what G_m$, defined above, on $\rn$ given by
\begin{equation*}
gx:=(\gamma z,y),\qquad\text{for all \ }\gamma\in\what G_m,\quad x=(z,y)\in\cc^2\times\r^{n-4}.
\end{equation*}
Let $\psi_m:\what G_m\to\z2$ be the restriction of $\phi_m$ to $\what G_m$. Then, \cite[Corollary 3.4]{c} establishes the existence of a $\psi_m$-equivariant solution $v_m\in D^{1,2}(\rn)$ to the Yamabe problem \eqref{yamabe rn} which satisfies $v_m(z,y)=v_m(z,y')$ if $|y|=|y'|$. This implies, in particular, that $\frac{\partial v_m}{\partial x_n}(x)=0$ for every $x\in\partial\hn$. Therefore, the restriction of $v_m$ to $\hn$ is a sign-changing solution to \eqref{problem_cap} with $b=0$. By Lemma \ref{lem:different}, all functions $v_1,\,v_2,\ldots,v_{2^k},\ldots$ are different.
\end{proof}

\bibliography{escobar}

@article{aw,
  author  = {Almaraz, Sérgio and Wang, Shaodong},
  title   = {Energy bounds of sign-changing solutions to {Y}amabe equations on manifolds with boundary},
  journal = {Nonlinear Analysis},
  volume  = {225},
  year    = {2022},
  pages   = {Paper No. 113131},
}

@article{aubin,
  author  = {Aubin, Thierry},
  title   = {Equations différentielles non linéaires et problème de {Y}amabe concernant la courbure scalaire},
  journal = {Journal de Mathématiques Pures et Appliquées (9)},
  volume  = {55},
  number  = {3},
  year    = {1976},
  pages   = {269--296},
}

@article{b,
  author  = {Beckner, William},
  title   = {Sharp {S}obolev inequalities on the sphere and the {M}oser--{T}rudinger inequality},
  journal = {Annals of Mathematics (2)},
  volume  = {138},
  number  = {1},
  year    = {1993},
  pages   = {213--242},
}

@article{c,
  author  = {Clapp, Mónica},
  title   = {Entire nodal solutions to the pure critical exponent problem arising from concentration},
  journal = {Journal of Differential Equations},
  volume  = {261},
  number  = {6},
  year    = {2016},
  pages   = {3042--3060},
}

@article{cfs,
  author  = {Clapp, Mónica and Faya, Jorge and Saldaña, Alberto},
  title   = {Optimal pinwheel partitions for the {Y}amabe equation},
  journal = {Nonlinearity},
  volume  = {37},
  number  = {10},
  year    = {2024},
  pages   = {Paper No. 105004},
}

@book{d,
  author    = {Deimling, Klaus},
  title     = {Ordinary Differential Equations in {B}anach Spaces},
  series    = {Lecture Notes in Mathematics},
  volume    = {596},
  publisher = {Springer-Verlag},
  address   = {Berlin--New York},
  year      = {1977},
}

@article{dmpp1,
  author  = {del Pino, Manuel and Musso, Monica and Pacard, Frank and Pistoia, Angela},
  title   = {Large energy entire solutions for the {Y}amabe equation},
  journal = {Journal of Differential Equations},
  volume  = {251},
  number  = {9},
  year    = {2011},
  pages   = {2568--2597},
}

@article{dmpp2,
  author  = {del Pino, Manuel and Musso, Monica and Pacard, Frank and Pistoia, Angela},
  title   = {Torus action on $\mathbb{S}^n$ and sign-changing solutions for conformally invariant equations},
  journal = {Annali della Scuola Normale Superiore di Pisa, Classe di Scienze (5)},
  volume  = {12},
  number  = {1},
  year    = {2013},
  pages   = {209--237},
}

@article{ding,
  author  = {Ding, Wei Yue},
  title   = {On a conformally invariant elliptic equation on $\mathbb{R}^n$},
  journal = {Communications in Mathematical Physics},
  volume  = {107},
  number  = {2},
  year    = {1986},
  pages   = {331--335},
}

@article{e1,
  author  = {Escobar, José F.},
  title   = {The {Y}amabe problem on manifolds with boundary},
  journal = {Journal of Differential Geometry},
  volume  = {35},
  number  = {1},
  year    = {1992},
  pages   = {21--84},
}

@article{e4,
  author  = {Escobar, José F.},
  title   = {Conformal deformation of a {R}iemannian metric to a constant scalar curvature metric with constant mean curvature on the boundary},
  journal = {Indiana Univ. Math. J.},
  volume  = {45},
  number  = {4},
  year    = {1996},
  pages   = {917–943},
}

@article{fp,
  author  = {Fernández, Juan Carlos and Petean, Jimmy},
  title   = {Low energy nodal solutions to the {Y}amabe equation},
  journal = {Journal of Differential Equations},
  volume  = {268},
  number  = {11},
  year    = {2020},
  pages   = {6576--6597},
}

@article{hl0,
  author  = {Han, Zheng-Chao and Li, YanYan},
  title   = {The {Y}amabe problem on manifolds with boundary: Existence and compactness results},
  journal = {Duke Mathematical Journal},
  volume  = {99},
  number  = {3},
  year    = {1999},
  pages   = {489--542},
}

@article{hl,
  author  = {Han, Zheng-Chao and Li, YanYan},
  title   = {The existence of conformal metrics with constant scalar curvature and constant boundary mean curvature},
  journal = {Communications in Analysis and Geometry},
  volume  = {8},
  number  = {4},
  year    = {2000},
  pages   = {809--869},
}

@article{h,
  author  = {Henry, Guillermo},
  title   = {Second {Y}amabe constant on Riemannian products},
  journal = {Journal of Geometry and Physics},
  volume  = {114},
  year    = {2017},
  pages   = {260--275},
}

@article{k,
  author  = {Kulpa, Władysław},
  title   = {The {P}oincaré--{M}iranda theorem},
  journal = {American Mathematical Monthly},
  volume  = {104},
  number  = {6},
  year    = {1997},
  pages   = {545--550},
}

@article{mm,
  author  = {Medina, Maria and Musso, Monica},
  title   = {Doubling nodal solutions to the {Y}amabe equation in $\mathbb{R}^n$ with maximal rank},
  journal = {Journal de Mathématiques Pures et Appliquées (9)},
  volume  = {152},
  year    = {2021},
  pages   = {145--188},
}

@article{p,
  author  = {Petean, Jimmy},
  title   = {On nodal solutions of the {Y}amabe equation on products},
  journal = {Journal of Geometry and Physics},
  volume  = {59},
  number  = {10},
  year    = {2009},
  pages   = {1395--1401},
}

@book{w,
  author    = {Willem, Michel},
  title     = {Minimax Theorems},
  series    = {Progress in Nonlinear Differential Equations and their Applications},
  volume    = {24},
  publisher = {Birkhäuser Boston},
  address   = {Boston, MA},
  year      = {1996},
}

@article{s,
  author  = {Struwe, Michael},
  title   = {A global compactness result for elliptic boundary value problems involving limiting nonlinearities},
  journal = {Mathematische Zeitschrift},
  volume  = {187},
  number  = {4},
  year    = {1984},
  pages   = {511–517},
}

@article{cpp,
  author  = {Clapp, Mónica and Pellacci, Benedetta and Pistoia, Angela},
  title   = {Sign-changing solutions to the {Y}amabe problem on manifolds with boundary},
  journal = {Journal or the London Mathematical Society},
  year  = {to appear},
}
\bibliographystyle{amsplain}

\end{document}